\documentclass[11pt]{amsart}

\usepackage{amsfonts, amstext, amsmath, amsthm, amscd, amssymb}
\usepackage{epsfig, graphics, psfrag, overpic, color, colonequals}
\let\oldmarginpar\marginpar
\renewcommand\marginpar[1]{\oldmarginpar[\raggedleft\footnotesize #1]%
{\raggedright\footnotesize #1}}

\renewcommand{\setminus}{{\smallsetminus}}

\newcommand{\ZZ}{{\mathbb{Z}}}

\newcommand{\abs}[1]{{\left\vert #1 \right\vert}}

\theoremstyle{plain}
\newtheorem{theorem}{Theorem}[section]
\newtheorem{corollary}[theorem]{Corollary}
\newtheorem{lemma}[theorem]{Lemma}

\newtheorem{prop}[theorem]{Proposition}

\newtheorem*{namedtheorem}{\theoremname}
\newcommand{\theoremname}{testing}
\newenvironment{named}[1]{\renewcommand{\theoremname}{#1}\begin{namedtheorem}}{\end{namedtheorem}}

\theoremstyle{definition}
\newtheorem{define}[theorem]{Definition}

\begin{document}
\title{Cosmetic crossings of genus one knots}
\author[C. Balm]{Cheryl balm}
\author[E. Kalfagianni]{Efstratia Kalfagianni}

\address[]{Department of Mathematics, Michigan State University, East
Lansing, MI, 48824}

\email[]{kalfagia@math.msu.edu}

\email[]{balmcher@math.msu.edu}

\thanks{ C. B. was supported in part by NSF grant DMS--0805942  and by 
NSF/RTG grants, DMS-0353717 and  DMS-0739208.}
\thanks{E. K. was supported in part by NSF grants DMS--0805942 and DMS-1105843.}
\thanks{ \today}

\begin{abstract} We show that for genus one knots the Alexander polynomial and the homology of the double cover branching over the knot
provide obstructions to cosmetic crossings. As an application we prove the nugatory crossing
conjecture for the negatively twisted, positive  Whitehead doubles of all knots. We also verify
the conjecture
for several families of pretzel knots and all genus one knots with up to 10 crossings.
\smallskip
\smallskip

%\smallskip
\smallskip

\smallskip
%\smallskip

\smallskip
\smallskip

\end{abstract}

\maketitle

%\newpage

%\tableofcontents

\section{Introduction}\label{sec:intro}
A fundamental open question in knot theory is the question of when a crossing 
change
on an oriented knot changes the isotopy class of the knot. A crossing disc for an oriented knot $K\subset S^3$
is an embedded disc $D\subset S^3$
such
that $K$ intersects ${\rm int}(D)$ twice with
zero algebraic intersection number. A crossing change on $K$ can be achieved
by twisting $D$ or equivalently by performing appropriate Dehn surgery of $S^3$
along the crossing circle $\partial D$.
The crossing is called nugatory if and only if
$\partial D$ bounds an embedded disc in the complement of $K$. A non-nugatory crossing on a knot $K$ is called
cosmetic if the oriented knot $K'$ obtained
from $K$ by changing the crossing  is isotopic to $K$. 
Clearly, changing a nugatory crossing doesn't
change the isotopy class of a knot. The nugatory crossing conjecture (Problem 1.58  of  Kirby's list \cite{Kirbylist}) asserts that  the converse is true: if a crossing change on a knot $K$ yields a knot
isotopic to $K$ then the crossing is
nugatory. In other words, there are not any knots in $S^3$ that admit cosmetic crossings.

In the case that $K$ is the trivial knot an affirmative answer follows from a result
of Gabai \cite{gabai} and work of Scharlemann and Thompson \cite{st}. The conjecture is also known  to hold for  2-bridge knots by work of Torisu \cite{torisu},
and for fibered knots by work  of Kalfagianni \cite{kalfagianni}.
For knots of braid index three a weaker form of the conjecture, requiring  that the crossing change  happens on a
closed 3-braid diagram, is discussed by Wiley in \cite{3braids}.

In this paper we study cosmetic crossings on genus one knots and we
show  that  the Alexander polynomial and the homology of the double cover branching over the knot
provide obstructions to cosmetic crossings.  Our main results are the following.

\begin{named}{Theorem \ref{aslice}} Let $K$ be an oriented genus one knot.
If $K$ admits a cosmetic crossing then it is algebraically slice. In particular,
 $\Delta_K(t) \doteq f(t) f(t^{-1})$,
where $f(t)\in \ZZ[t]$ is a linear polynomial.
\end{named}

\begin{named}{Theorem \ref{doublecover}} Let $K$ be an oriented genus one knot and let $Y_K$  denote the double
cover of $S^3$ branching over $K$.
If $K$ admits a cosmetic crossing then the homology group $H_1(Y_K):=H_1(Y_K, {\ZZ})$ is  a finite cyclic group.
\end{named}
 
 Given a knot $K$ let $D_{+}(K, n)$ denote the $n$-twisted, positive clasped Whitehead double of $K$.
 As an application of 
Theorems \ref{aslice}  and \ref{doublecover} we prove that, for every $K$ and $n<0$,  the knot $D_{+}(K, n)$ satisfies the nugatory crossing conjecture (Corollary \ref{whitehead}).
We also prove the conjecture for several families of pretzel knots (Corollary \ref{pretzels}) and 22 out of the 23 genus one knots with
up to twelve crossings (see Section 5).

 \vskip 0.04in
 
 Throughout the paper we will discuss oriented knots in an oriented $S^3$ and we work in the smooth category.
\smallskip

{\bf Acknowledgements.} 
We thank Matt Hedden and Matt Rathbun for helpful discussions.
\smallskip

\section{Crossing changes and arcs on surfaces}
 In this section we use a result of Gabai \cite{gabai} to prove that a cosmetic crossing change on a knot $K$ can be realized by twisting along an
 essential arc on a minimum genus Seifert surface of $K$ (Proposition \ref{prop:minimum}). For genus one knots such an arc will be non-separating on the surface.
 In the next sections this will be our
 starting point for establishing connections between cosmetic crossings and knot invariants determined by Seifert matrices.
 \vskip 0.1in
 
Let $K$ be an oriented  knot in $S^3$ and $C$ be a crossing
of sign $\epsilon$, where $\epsilon=1$ or $-1$ according to whether $C$ is a 
positive or negative crossing (see Figure 1).
A \emph{crossing disc} of $K$ corresponding to $C$
is an embedded disc $D\subset S^3$
such
that $K$ intersects ${\rm int}(D)$ twice, once for each branch of $C$, with
zero algebraic intersection number. The boundary $L = \partial D$ is called a
\emph{crossing circle}.
Performing $({\textstyle {{-\epsilon}}})$-surgery on $L$, 
changes $K$ to another knot $K^{'}\subset S^3$ that
is obtained from $K$
by changing the crossing $C$.  

\begin{define} \label{nugat} A crossing  supported on a crossing circle $L$
of an oriented knot $K$ is called  \emph{nugatory} if 
$L = \partial D$ also bounds an embedded disc in the complement of $K$. This disc and $D$ form an
embedded
2-sphere that decomposes $K$ into a connected sum
where some of the summands may be trivial.
A non-nugatory crossing on a knot $K$ is called
\emph{cosmetic} if the oriented knot $K'$ obtained
from $K$ by changing $C$ is isotopic to $K$; that is, there exists an orientation-preserving diffeomorphism
$f: S^3\longrightarrow S^3$ with $f(K)=K'$.
\end{define}
\vskip .1in

\begin{figure}
\includegraphics[width=1.3in]{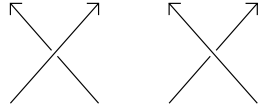}
\caption{Left: a positive crossing. Right: a negative crossing.}\label{cr_signs}
\end{figure}

For a link $J$ in $S^3$ we will use $\eta(J)$ to denote a regular neighborhood of $J$ in $S^3$
and we will use $M_J \colonequals \overline {S^3\setminus \eta(J)}$ to denote the closure of the complement of $\eta(J)$ in $S^3$.

\begin{lemma} \label{lem:irreducible} Let $K$ be an oriented knot and $L$ a crossing circle supporting a crossing $C$ of $K$.
Suppose that $M_{K \cup L}$ is reducible. Then $C$ is nugatory.

\end{lemma}
\begin{proof} An essential 2-sphere in $M_{K \cup L}$ must separate
$\eta(K)$ and $\eta(L)$. Thus in $S^3$,  $L$ lies in a 3-ball disjoint from $K$.
Since $L$ is unknotted, it bounds a disc in the complement of $K$.
 \end{proof}
%\vskip 0.07in

Let $K$ be an oriented knot and $L = \partial D$ a crossing circle supporting a  crossing $C$.
Let $K'$ denote the knot obtained from $K$ by changing $C$.
Since the linking number of $L$ and $K$ is zero, $K$ bounds a Seifert surface in the complement of $L$. Let
$S$ be a Seifert surface that is of minimum genus among all such Seifert surfaces in the complement of $L$. Since $S$ is incompressible, after an isotopy we can arrange so that the closed components of  $S\cap D$ are homotopically essential in $D\setminus K$.
But then each such component is parallel to $\partial D$ on $D$  and by further modification we can arrange so that
$S\cap D$ is a single arc $\alpha$  that is properly embedded on $S$ as illustrated in Figure \ref{alpha}.
The surface
$S$
gives rise
to Seifert surfaces $S$ and $S'$ of $K$ and $K'$, respectively. 
\begin{figure}
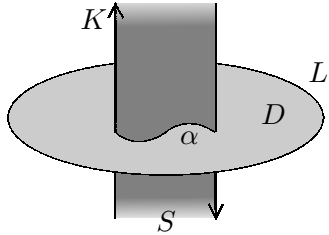
\caption{The crossing arc $\alpha = S \cap D$.}\label{alpha}
\end{figure}

\begin{prop} \label{prop:minimum} 
Suppose that $K$ is isotopic to $K'$. Then $S$ and $S'$ are Seifert surfaces of minimal genus for $K$ and $K'$, respectively.
\end{prop}
\begin{proof} If the crossing is nugatory then $L$ bounds a disc in the complement of $S$ and the conclusion is clear.
Suppose the crossing is cosmetic; by Lemma \ref{lem:irreducible},
$M_{K \cup L}$ is irreducible.
We can consider the surface $S$ properly embedded in $M_{K \cup L}$
so that it is disjoint from 
$\partial \eta(L) \subset \partial M$.   
The assumptions
on irreducibility of $M_{K \cup L}$ and 
on the genus of $S$ imply that the foliation machinery of Gabai \cite{gabai} applies. In particular,
$S$ is taut in the Thurston norm. The manifolds
$M_K$ and $M_{K'}$ are obtained by Dehn filling of $M_{K \cup L}$ along $\partial \eta(L)$.
By \cite[Corollary 2.4] {gabai}, $S$ can fail to remain taut in the Thurston norm
(i.e. genus minimizing)
in at most one of $M_K$ and $M_{K'}$.   Since we have assumed that $C$ is a cosmetic crossing,  $M_K$ and $M_{K'}$
are homeomorphic (by an orientation preserving homeomorphism). Thus $S$  
remains taut in both of $M_K$ and $M_{K'}$. This implies that $S$
and $S'$ are Seifert surfaces of minimal genus for $K$ and $K'$, respectively.
\end{proof}

By Proposition \ref{prop:minimum}, a crossing change of a knot $K$ that produces an isotopic knot
corresponds to a properly embedded arc $\alpha$ on a minimum genus Seifert surface $S$ of $K$.
We observe the following.
\begin{lemma} \label{lem:essential} If $\alpha$ is inessential on $S$, then the crossing is nugatory.
\end{lemma}
\begin{proof} Recall that $\alpha$ is the intersection of a crossing disc $D$ with $S$.
Since $\alpha$ is inessential, it separates $S$ into two pieces, one of which is a disc $E$.
Consider $D$ as properly embedded in a regular neighborhood $\eta(S)$ of the surface $S$.
The boundary of a regular neighborhood of $E$ in $\eta (S)$ is a 2-sphere
that contains the crossing disc $D$. The complement of the interior of $D$ in that 2-sphere gives a disc
bounded by the crossing circle $L=\partial D$ with its interior disjoint from the  knot $K=\partial S$. 
\end{proof}

\section{Obstructing cosmetic crossings in genus one knots}
A  knot  $K$ is called \emph{algebraically slice}
if it admits a Seifert surface $S$ such that the Seifert form
$\theta: H_1(S)\times H_1(S)\longrightarrow \ZZ$
vanishes on a half-dimensional summand of $H_1(S)$; such a summand is called a \emph{metabolizer} of $H_1(S)$.
If $S$ has genus one, then the existence  of a metabolizer for $H_1(S)$ is equivalent to the existence
of an essential oriented simple closed curve on $S$ that has
zero self-linking number. If $K$ is algebraically slice, then the Alexander polynomial
$\Delta_K(t)$ is of the form $\Delta_K(t) \doteq f(t) f(t^{-1})$,
where $f(t)\in \ZZ[t]$ is a linear polynomial with integer coefficients
and $\doteq$ denotes equality up to multiplication by a unit in the ring of Laurent polynomials ${\ZZ}[t, \ t^{-1}]$.
For more details on these and other classical knot theory concepts we will use in this and the next section, the reader is referred to 
\cite{burde-zieschang:knots} or  \cite{lickorish:book}.

\begin{theorem}\label{aslice} Let $K$ be an oriented genus one knot.
If $K$ admits a cosmetic crossing, then it is algebraically slice. In particular,
there is a   linear polynomial $f(t)\in \ZZ[t]$ such that
 $\Delta_K(t)\doteq f(t) f(t^{-1})$.

\end{theorem}
\begin{proof} Let $K'$ be a knot that is obtained from $K$ by a cosmetic crossing change $C$.
By Proposition \ref{prop:minimum}, there is a genus one Seifert surface $S$ 
such that a crossing disc supporting $C$ intersects $S$ in a properly embedded arc $\alpha \subset S$.
Let $S'$ denote the result of $S$ after the crossing change. 
 \begin{figure}
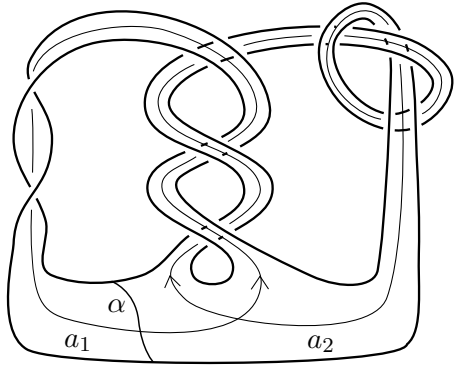
\caption{A genus one surface $S$ with generators $a_1$ and $a_2$ of $H_1 (S)$ and a non-separating arc $\alpha$.}\label{fig1}
\end{figure}
Since $C$ is a cosmetic crossing, by Lemma \ref{lem:essential},  $\alpha$ is essential.  Further, since the genus of $S$ is one, $\alpha$ is non-separating.
We can find a simple closed curve $a_1$ on $S$ that intersects $\alpha$ exactly once. Let $a_2$ be another simple closed curve so that
$a_1$ and $a_2$ intersect exactly once and the homology classes of $a_1$ and $a_2$ form a  symplectic basis for
 $H_1 (S) \cong {\ZZ} \oplus {\ZZ}$.  Note that $\{ a_1, a_2 \}$ form a corresponding basis of $H_1(S')$. See Figure \ref{fig1}.
 
 The Seifert matrices of $S$  and $S'$ with respect to these bases are
 $$V = \begin{pmatrix}a & b \\ c & d\end{pmatrix} \ \  \ \rm{and}  \ \ \   V'= \begin{pmatrix}a-\epsilon & b \\ c & d\end{pmatrix}$$
respectively, where $a,b,c, d\in {\ZZ}$ and $\epsilon=1$ or $-1$ according to whether $C$ is a positive or a negative crossing.
 The Alexander polynomials of $K$,  $K'$ are given by
 $$\Delta_K(t)\doteq \det( V-tV^T)=ad(1-t)^2-(b-ct)(c-tb),$$
 $$\Delta_{K'}(t)\doteq(a-\epsilon)d(1-t)^2-(b-ct)(c-tb).$$
Since $K$ and $K'$ are isotopic we must have $\Delta_K(t)\doteq\Delta_{K'}(t)$ which easily implies that $d={\rm lk}(a_2, \ a_2)=0$.
Hence $K$ is algebraically slice and
 $$\Delta_K(t)\doteq(b-ct)(c-tb)=(-t)(b-ct)(b-ct^{-1})\doteq (b-ct)(b-ct^{-1}).$$
 Setting $f(t)=b-ct$ we obtain $\Delta_K(t)\doteq f(t) f(t^{-1})$ as desired.  Note that since $\abs{b-c}$ is the intersection number between $a_1$ and $a_2$, by suitable orientation choices, we may assume that $c=b+1$.
\end{proof}

Recall that the determinant of a knot $K$ is defined by $\det(K)=\abs{\Delta_K(-1)}$.
As a corollary of Theorem \ref{aslice} we have the following.

\begin{corollary}\label{determinant} Let $K$ be a genus one knot. If $\det(K)$ is not a perfect square then
$K$ admits no cosmetic crossings.
\end{corollary}
\begin{proof}  Suppose that $K$ admits a cosmetic crossing. By Theorem \ref{aslice}
$\Delta_K(t)\doteq f(t) f(t^{-1})$,
where $f(t)\in \ZZ[t]$ is a linear polynomial. Thus, if $K$ admits cosmetic crossings we have
 $\det(K)=\abs{\Delta_K(-1)}=[f(-1)]^{2}$.
\end{proof}

\section{Further obstructions: homology of double covers}
Here we will derive further obstructions to cosmetic crossings in terms of the homology of the double branched cover
of the knot. Given $m\in {\ZZ}$ let ${\ZZ}_m={\ZZ}/{m{\ZZ}}$ denote the cyclic abelian group of order $\abs{m}$.
The following lemma will be useful to us.
\begin{lemma} \label{abelian} If $H$ denotes the abelian group given by the presentation
  \begin{equation*}
H\cong \left<
\begin{array}{l|l}
c_1, c_2 & 2x c_1+ (2y+1)c_2= 0\\
& (2y+1)c_1= 0 
\end{array}
\right> ,
\end{equation*}
then we have
\begin{enumerate}
\item $H\cong 0$, if $y=0$ or $y=-1$.
\item  $H \cong \mathbb{Z}_{d} \oplus \mathbb{Z}_{\frac{{(2y+1)}^2}{d}}$, if   $y \neq 0,\ -1$ and {\rm gcd}$(2x,\ 2y+1) = d$ where $1 \leq  d \leq 2y+1$.
\end{enumerate}
  \end{lemma}
  \begin{proof} The proof is an easy linear algebra exercise  left to the reader.
  \end{proof}

\begin{theorem}\label{doublecover} Let $K$ be an oriented genus one knot and let $Y_K$  denote the double
cover of $S^3$ branching over $K$.
If $K$ admits a cosmetic crossing, then the homology group $H_1(Y_K)$ is  a finite cyclic group.
\end{theorem}
\begin{proof}
  Suppose that a genus one knot $K$ admits a cosmetic crossing yielding an isotopic knot  $K'$.
The proof of Theorem \ref{aslice} shows that $K$ and $K'$ admit Seifert matrices
of the form
 $$V = \begin{pmatrix}a & b\\ b+1 & 0\end{pmatrix} \ \  \ \rm{and}  \ \ \   V'= \begin{pmatrix}a-\epsilon & b \\ b+1 & 0\end{pmatrix} \eqno(1)$$
respectively, where $a,b \in {\ZZ}$ and $\epsilon=1$ or $-1$ according to whether $C$ is a positive or a negative crossing.
 In particular we have
 $$\Delta_K(t)\doteq \Delta_{K'}(t)\doteq b(b+1)(t^2+1)-(b^2+(b+1)^2)t. \eqno(2)$$
 Presentation matrices for $H_1(Y_K)$ and $H_1(Y_{K'})$ are given by
 $$V+V^T = \begin{pmatrix}2a & 2b+1\\ 2b+1 & 0\end{pmatrix}, \ \  \ \rm{and}  \ \ \  V'+  (V')^{T}= \begin{pmatrix} 2a-2\epsilon &2 b+1 \\ 2b+1 & 0\end{pmatrix}\eqno(3)$$
 respectively.
It follows that Lemma \ref{abelian} applies to both $H_1(Y_K)$ and $H_1(Y_{K'})$. By that lemma,
$H_1(Y_K)$ is either cyclic or $H_1(Y_K)\cong\mathbb{Z}_{d} \oplus \mathbb{Z}_{\frac{{(2b+1)}^2}{d}}$, 
with   $b \neq 0,\ -1$ and $ {\rm gcd}(2a,\ 2b+1) = d$ where $1 < d \leq 2b+1$.
Similarly, $H_1(Y_{K'})$ is either cyclic or $H_1(Y'_K)\cong\mathbb{Z}_{d'} \oplus \mathbb{Z}_{\frac{{(2b+1)}^2}{d'}}$,
with $ {\rm gcd}(2a-2\epsilon,\ 2b+1) = d'$ where $1 < d' \leq 2b+1$.
Since $K$ and $K'$ are isotopic, we have $H_1(Y_K)\cong H_1(Y_{K'})$. One can easily verify this
can only happen in the case that $\textrm{gcd}(2a,\ 2b+1) = \textrm{gcd}(2a-2\epsilon,\ 2b+1) = 1$ in which case $H_1 (Y_K)$ is cyclic.
 \end{proof}
 
 It is known that for an algebraically slice knot of genus one every minimum genus surface $S$ contains a metabolizer
 (compare \cite[Theorem 4.2]{livingston})
 After completing the metabolizer  to a basis of $H_1(S)$ we have a Seifert matrix $V$ as in (1) above.

 \begin{corollary} \label{matrix} Let $K$ be an oriented, algebraically slice knot of genus one.
 Suppose that a genus one Seifert surface of $K$ contains a metabolizer leading to a Seifert matrix $V$ as in (1)
 so that  $b\neq 0,-1$ and $\emph{gcd}(2a,\ 2b+1) \neq 1$.
 Then $K$ cannot admit a cosmetic crossing.
 \end{corollary} 
 
\begin{proof} Let $d=\textrm{gcd}(2a,\ 2b+1) $. As in the proof of Theorem \ref{doublecover},
 we use Lemma \ref{abelian}  to conclude that  $H_1(Y_K)\cong\mathbb{Z}_{d} \oplus \mathbb{Z}_{\frac{{(2b+1)}^2}{d}}$
 and hence is non-cyclic unless $d=1$. Now the conclusion follows by Theorem \ref{doublecover}.
  \end{proof}

 \section{Applications}
\subsection{Twisted Whitehead doubles} Given a knot $K$ let $D_{+}(K, n)$ denote the $n$-twisted
Whitehead double of $K$ with a positive clasp and let  $D_{-}(K, n)$ denote the $n$-twisted
Whitehead double of $K$ with a negative clasp.

\begin{corollary}\label{whitehead}
Given a knot $K$,  the Whitehead double $D_{+}(K, n)$ admits no cosmetic crossing if either $n<0$ or $\abs{n}$ is odd.  Similarly  $D_{-}(K, n)$ admits no cosmetic crossing if either $n>0$ or $\abs{n}$ is odd.
\end{corollary}

\begin{proof} A Seifert surface of $D_{+}(K, n)$ obtained by pluming an $n$-twisted annulus 
with core $K$ and a Hopf band gives rise to a  Seifert matrix
$V_n = \begin{pmatrix}-1& 0 \\ -1 & n\end{pmatrix}$  \cite[Example 6.8]{lickorish:book}.
Thus the Alexander polynomial
is
of the form
$$\Delta_n\doteq -n(t^2+1)+(1+2n)t\eqno(4)$$
Suppose now that 
$D_{+}(K, n)$ admits a cosmetic crossing.  Then
$\Delta_n$ should be of the form shown in equation (2).
Comparing the leading coefficients in the expressions (2) and (4)
we obtain $\abs{n}=\abs{b(b+1)}$ which implies that $\abs{n}$ should be even.
We have shown that if $\abs{n}$ is odd then $D_{+}(K, n)$ admits no
cosmetic crossing changes. 
Suppose now that $n<0$. Since the Seifert matrix $V_n$ 
only depends only on $n$ and not on $K$, $D_{+}(K, n)$ is $S$-equivalent
to the positively clasped, $n$-twisted double of the unknot. This is a positive knot
and it has non-zero signature \cite{positive}. Hence $D_{+}(K, n)$ 
is not algebraically slice and by Theorem \ref{aslice} it cannot admit
cosmetic crossings.

A similar argument holds for $D_{-}(K, n)$.
\end{proof}
\begin{figure}
  \includegraphics[width=1.5in]{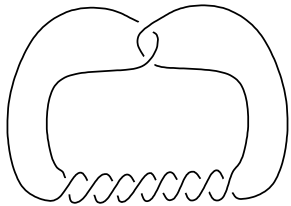}
  \caption{The  $(-4)$-twisted  negative clasped double of the unknot $D_{-}(O, -4)$.}
  \label{fig:twistknot}
\end{figure}

\subsection{Pretzel knots}
Let $K$ be a three string  pretzel knot $P(p,q,r)$ with $p,q$ and $r$  odd (see Figure \ref{pretzelknots}).  The knot  determinant is given
by $\det (K) = |pq+qr+pr|$ and if $K$ is non-trivial then it has genus one.
It is known that  $K$ is algebraically  slice  if and only if $pq+qr+pr=-m^2$, for some odd  $m\in {\ZZ}$ \cite{levine}.

\begin{corollary} \label{pretzels}The knot  $P(p,q,r)$ with $p,q$ and $r$ odd does not admit cosmetic crossings if one of the following is true:
\begin{enumerate}
\item $pq+qr+pr\neq -m^2$, for every odd $m\in {\ZZ}$.
\item $q+r=0$ and $\textrm{gcd}(p,\  q)\neq 1$.
\item $p+q=0$ and $\textrm{gcd}(p,\  r)\neq 1$.
\end{enumerate}
\end{corollary}

\begin{proof} In case (1) the result follows from Theorem \ref{aslice} and the discussion above.
For case (2) recall that there is a genus one surface for $P(p,q,r)$ for which
a Seifert matrix is $V_{(p,q,r)} = \frac{1}{2}\begin{pmatrix}{p+q} & {q+1} \\ {q-1} & q+r\end{pmatrix}$ \cite[Example 6.9]{lickorish:book}.
Suppose that   $q+r=0$. If  $\textrm{gcd}(p,\  q)\neq 1$,
then  $\textrm{gcd}(p+q,\  q)\neq 1$ and the conclusion in case (2) follows by
Corollary \ref{matrix}. Case (3) is similar.
\end{proof}

\begin{figure}
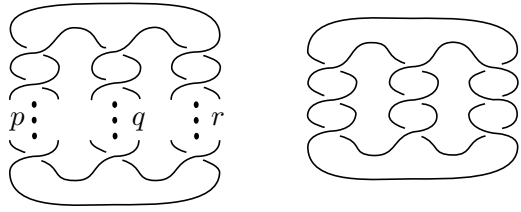
\caption{$P(p,q,r)$ with $p,q$ and $r$ positive and $ P(3,3,-3)$.}
\label{pretzelknots}
\end{figure}
\begin{table}[height=1.00in]
\begin{center}
\begin{tabular}{|c|c|c|c|c|c|}
\hline
$K$& $\det (K)$ & $K$ & $\det (K)$& $K$ & $\det (K)$  \\
\hline
$3_1$ &3& $9_2$ &15& ${11\textrm{a}}_{362}$ & 39 \\
\hline
$4_1$ & 5& $9_5$ & 23&${11\textrm{a}}_{363}$&35 \\
\hline
$5_2$ &7&$9_{35}$ &27&${\bf {11\textrm{\bf n}}_{139}}$ &{\bf 9}\\
\hline
${\bf 6_1}$ &{\bf 9}&${ \bf 9_{46}}$ &{\bf 9}&${11\textrm{n}}_{141}$& 21\\
\hline
$7_2 $& 11&$10_1$&17&${12\textrm{a}}_{803}$ & 21\\
\hline
$7_4$&15&${\bf 10_3}$ & {\bf 25}&${12\textrm{a}}_{1287}$ & 37\\
\hline
$8_1$ &13&${11\textrm{a}}_{247}$ &19&${12\textrm{a}}_{1166}$& 33\\
\hline
$8_3$&17&${11\textrm{a}}_{343}$ & 31&-&- \\
\hline
\end{tabular}
\label{twelve}
\end{center}
\vskip.13in
\caption{Genus one knots with at most 12 crossings.}
\end{table}

\subsection{Low crossing knots} 
Table 1,  obtained from KnotInfo \cite{knotinfo}, gives  the 23 knots of genus one with at most 12 crossings
with the values of their determinants.
There are four knots with square determinant. Thus
Corollary  \ref{determinant} excludes 
cosmetic crossings for all but the knots
$6_1$,  $9_{46},  10_3$ and $11\textrm{n}_{139}$ on Table 1. 
Now  $6_1$ and $10_3$ are 2-bridge knots and by \cite{torisu} they don't admit cosmetic crossings.
The knot
$9_{46}$  is isotopic to  $P(3,3,-3)$ of Figure \ref{pretzelknots} which
by Corollary 
\ref{pretzels}
 cannot have cosmetic crossings. 
 The knot  $K=11\textrm{n}_{139}$ is isotopic to 
 $P(-5,3,-3)$.
 As in the proof of Corollary
 \ref{pretzels} we calculate $H_1(Y_K)\cong {\ZZ}_9$.
 Thus Theorems \ref{aslice} and \ref{doublecover} fail to settle
 the nugatory crossing conjecture for $11\textrm{n}_{139}$. 
 Since $K$ is neither fibered nor a 2-bridge knot the results of \cite{kalfagianni, torisu}
 also fail to settle the conjecture.

  \bibliographystyle{hamsplain}
\bibliography{biblio}

\end{document}